\documentclass[12pt]{amsart}

\usepackage{lineno}
\usepackage{hyperref} 
\usepackage{amsthm,amsmath,amssymb}
\usepackage{graphicx}
\usepackage[utf8]{inputenc}
\usepackage[T1]{fontenc}
\usepackage[toc,page]{appendix} 
\usepackage[english,french]{babel}
\usepackage{enumitem}
\usepackage{amsfonts}
\usepackage{mathtext}
\usepackage{array}
\usepackage{multirow}
\usepackage[font=footnotesize]{caption}
\usepackage{caption}
\usepackage{subcaption}
\usepackage{yhmath}
\usepackage{mathdots}
\usepackage{tikz}
\usepackage{placeins}

\newtheorem{theo}{Theorem}

\newtheorem{prop}[theo]{Proposition}

{\theoremstyle{definition} \newtheorem{defi}[theo]{Definition}}
{\theoremstyle{definition} }
{\theoremstyle{definition} }
{\theoremstyle{definition} \newtheorem*{nota}{Notation}}
{\theoremstyle{remark} \newtheorem{rem}[theo]{Remark}}

\title[Combinatorial interpretations of the Kreweras triangle]{Combinatorial interpretations of the Kreweras triangle in terms of subset tuples}
\author{Ange Bigeni}
\thanks{National Research University Higher School of Economics, Faculty of Mathematics, Usacheva
str. 6, 119048, Moscow, Russia. \href{mailto:abigeni@hse.ru}{abigeni@hse.ru}}

\begin{document}

\selectlanguage{english}

\maketitle

\begin{abstract}
We show how the combinatorial interpretation of the normalized median Genocchi numbers in terms of multiset tuples, defined by Hetyei in his study of the alternation acyclic tournaments, is bijectively equivalent to previous models like the normalized Dumont permutations or the Dellac configurations, and we extend the interpretation to the Kreweras triangle.
\end{abstract}

\section*{Notations}

For all pair of integers $n < m$, the set $\{n,n+1,\hdots,m\}$ is denoted by $[n,m]$, and the set $[1,n]$ by $[n]$. The set of the permutations of $[n]$ is denoted by $\mathfrak{S}_n$.

\section{Introduction}
\subsection{Genocchi numbers, Kreweras triangle, Dumont permutations}
The Genocchi numbers $(G_{2n})_{n \geq 1} = (1,1,3,17,155,2073,\hdots)$~\cite{G} and median Genocchi numbers $(H_{2n+1})_{n \geq 0} = (1,2,8,56,608,\hdots)$~\cite{H} can be defined as the positive integers $G_{2n} = g_{2n-1,n}$ and $H_{2n+1} = g_{2n+2,1}$~\cite{DV} where $(g_{i,j})_{1 \leq j \leq i}$ is the Seidel triangle defined by 
\begin{align*}
g_{2p-1,j} &= g_{2p-1,j-1} + g_{2p-2,j},\\
g_{2p,j} &= g_{2p-1,j} + g_{2p,j+1},
\end{align*}
with $g_{1,1} = 1$ and $g_{i,j} = 0$ if $i < j$. It is well known that $H_{2n+1}$ is divisible by $2^n$ for all $n \geq 0$~\cite{Barsky}. The normalized median Genocchi numbers $(h_n)_{n \geq 0} = (1, 1, 2, 7, 38, 295, \hdots)$~\cite{h} are the positive integers defined by 
$$h_n = H_{2n+1}/2^n.$$
Dumont~\cite{Dumont} gave the first combinatorial models of the (median) Genocchi numbers. In particular, the set $PD2_n$ of the Dumont permutations of the second kind, that is, the permutations $\sigma \in \mathfrak{S}_{2n+2}$ such that $\sigma(2i-1) > 2i-1$ and $\sigma(2i) < 2i$ for all $i \in [n+1]$, whose cardinality $\# PD2_n$ equals $H_{2n+1}$ for all $n \geq 0$. In \cite{Kreweras}, Kreweras introduced the subset $PD2N_n \subset PD2_n$ of the normalized such permutations, \textit{i.e.}, the permutations $\sigma \in PD2_n$ such that $\sigma^{-1}(2i) < \sigma^{-1}(2i+1)$ for all $i \in [n]$, whose number is $\# PD2N_n = h_n$.

\begin{rem}
For all $(k,l) \in [n]^2$, let $PD2N_{n,k}$ (respectively $PD2N'_{n,l}$) be the subset of the $\sigma \in PD2N_n$ such that $\sigma(1) = 2k$ (respectively $\sigma(2n+2) = 2l+1$). It is easy to see that $\{ PD2N_{n,k} : k \in [n]\}$ and $\{PD2N'_{n,l} : l \in [n]\}$ are partitions of $PD2N_n$.
\end{rem}

In \cite{KB}, by introducing the model of the alternating diagrams and connecting them bijectively to the normalized Dumont permutations, Kreweras and Barraud proved that $$\# PD2N_{n,k} = \# PD2N'_{n,k} = h_{n,k}$$ where the Kreweras triangle $(h_{n,k})_{n \geq 1,k \in [n]}$~\cite{Kreweras} (see Figure~\ref{fig:krewerastriangle}) is defined by $h_{1,1} = 1$ and, for all $n \geq 2$ and $k \in [3,n]$, 
\begin{align}
\label{eq:definitionhnk}
h_{n,1} &= h_{n-1,1} + h_{n-1,2} + \hdots + h_{n-1,n-1}, \nonumber \\
h_{n,2} &= 2 h_{n,1} - h_{n-1,1},\\
h_{n,k} &= 2h_{n,k-1} - h_{n,k-2} - h_{n-1,k-1} - h_{n-1,k-2}.\nonumber
\end{align}

\begin{figure}[!h]
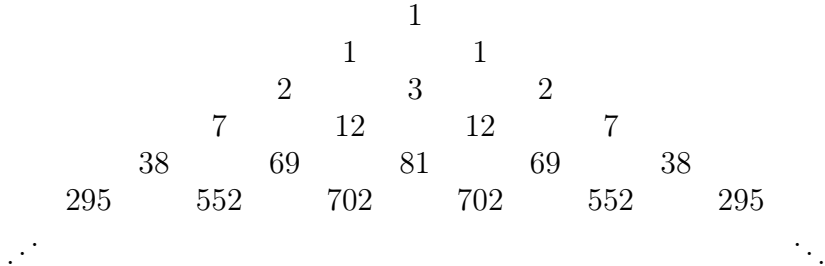

$$\begin{tabular}{ccccccccccccc}
& & & & & & 1 & & & & & \\
 && & & & 1 & & 1 & & & &  \\
 && & & 2 & & 3 & & 2 & & & \\
 & && 7 & & 12 & & 12 & & 7  \\
& & 38 & & 69 & & 81 & & 69 & & 38\\
& 295 & & 552 & & 702 & & 702 & & 552 & & 295\\
$\iddots$ & & &&&&&&&& & & $\ddots$
\end{tabular}$$
\caption{The Kreweras triangle.}
\label{fig:krewerastriangle}
\end{figure}

For example, we depict in Figure \ref{fig:repartitionPD2N3} how are partitionned the $h_3 = 2+3+2$ elements of~$PD2N_3$.

\begin{figure}[h]
\begin{center}

\begin{tikzpicture}[scale=1]

\draw (-4.3,3.5) -- (-4.3,-0);
\draw (-0.9,3.5) -- (-0.9,-0);
\draw (2.85,3.5) -- (2.85,-0);

\draw (-6,2.5) -- (6.5,2.5);
\draw (-6,1.5) -- (6.5,1.5);
\draw (-6,0.5) -- (6.5,0.5);

\draw (1,3) node[scale=1] {$21637485$};
\draw (4.75,3) node[scale=1] {$21436587$};
\draw (-2.75,2) node[scale=1] {$41627583$};
\draw (1,2) node[scale=1] {$41627385$};
\draw (4.75,2) node[scale=1] {$41526387$};
\draw (-2.75,1) node[scale=1] {$61427583$};
\draw (1,1) node[scale=1] {$61427385$};

\draw (-5.2,3) node[scale=1] {$PD2N_{3,1}$};
\draw (-5.2,2) node[scale=1] {$PD2N_{3,2}$};
\draw (-5.2,1) node[scale=1] {$PD2N_{3,3}$};

\draw (-2.75,0) node[scale=1] {$PD2N'_{3,1}$};
\draw (1,0) node[scale=1] {$PD2N'_{3,2}$};
\draw (4.75,0) node[scale=1] {$PD2N'_{3,3}$};

\end{tikzpicture}
\end{center}
\caption{The partition of $PD2N_3$.}
\label{fig:repartitionPD2N3}
\end{figure}

For all $n \geq 1$ and $k \in [n]$, the Kreweras triangle has the visible two properties
\begin{align}
h_{n,n} &= h_{n-1}, \label{eq:hn1}\\
h_{n,k} &= h_{n,n-k+1},
\label{eq:symmetryofhnk}
\end{align}
of which \cite{KB} implies interpretations in terms of $PD2N_n$. Formula~(\ref{eq:hn1}) follows from the bijection $\sigma \in PD2N'_{n,n} \mapsto \sigma_{|[2n]} \in PD2N_{n-1}$. Afterwards, let $\sigma \in PD2N_n$ and $(k,l) \in [n]^2$ such that $\sigma(1) = 2k$ and $\sigma(2n+2) = 2l+1$, we define two permutations $\sigma^t$ and $\sigma^r$ as follows.
\begin{itemize}
\item If $k = l$, we define $\sigma^t$ as $\sigma$, otherwise it is defined as the composition $\begin{pmatrix}
2k & 2l & 2l+1 & 2k+1
\end{pmatrix} \circ \sigma$.
\item We define $\sigma^r$ by $\sigma^r(i) = 2n+3 - \sigma(2n+3-i)$ for all $i \in [2n+2]$.
\end{itemize}
The maps $\sigma \mapsto \sigma^t$ and $\sigma \mapsto \sigma^r$ are involutions of $PD2N_n$ which induce bijections 
\begin{align*}
PD2N_{n,k} \cap PD2N'_{n,l} &\longleftrightarrow PD2N_{n,l} \cap PD2N'_{n,k},\\
PD2N_{n,l} \cap PD2N'_{n,k} &\longleftrightarrow PD2N_{n,n-k+1} \cap PD2N'_{n,n+1-l},
\end{align*}
from which follows Formula (\ref{eq:symmetryofhnk}). One can also obtain it by induction from System (\ref{eq:definitionhnk}) through the easy equality
$$h_{n,k} - h_{n,k-1} = \sum_{i = k}^{n-1} h_{n-1,i} - \sum_{i=1}^{k-2} h_{n-1,i}$$
for all $n \geq 1$ and $k \in [n]$ (where $h_{n,0}$ is defined as $0$).

There are several other bijectively equivalent models of the Kreweras triangle~\cite{Dellac,kreweras2,HZ,Feigin2,Bigeni}. 

\subsection{The Dellac configurations}
The Dellac configurations~\cite{Dellac} form the earliest combinatorial model of the Kreweras triangle and provide a geometrical analogous of the previous results. Recall that a Dellac configuration of size $n$ is a tableau $D$, made of $n$ columns and $2n$ rows, that contains $2n$ dots such that~:
\begin{itemize}
\item every row contains exactly one dot;
\item every column contains exactly two dots;
\item if there is a dot in the box $(j,i)$ of $D$ (\textit{i.e.}, in the intersection of its $j$-th column from left to right and its $i$-th row from bottom to top), then $j \leq i \leq j+n$.
\end{itemize}

The set of the Dellac configurations of size $n$ is denoted by $DC_n$. It can be partitionned into $\{DC_{n,k} : k \in [n]\}$ or $\{DC'_{n,l} : l \in [n]\}$ where $DC_{n,k}$ (respectively $DC'_{n,l}$) is the subset of the tableaux $D \in DC_n$ whose box $(k,n+1)$ (respectively $(l,n)$) contains a dot, for all $(k,l) \in [n]^2$. In \cite[Proposition 3.3]{Feigin}, Feigin constructs a bijection $f_1 : PD2N_n \rightarrow DC_n$ such that $f_1(PD2N_{n,k}) = DC_{n,k}$, hence $h_{n,k} = \# DC_{n,k}$, for all $k \in [n]$. One can also check that $f_1(PD2N'_{n,k}) = DC'_{n,k}$, so $h_{n,k} = \# DC'_{n,k}$. For example, the $h_3 = 2+3+2$ elements of $DC_3$ are partitionned as depicted in Figure \ref{fig:repartitionDC3}.

\begin{figure}[h]
\begin{center}

\begin{tikzpicture}[scale=0.8]

\draw (-3,6) -- (-3,-5);
\draw (-1,6) -- (-1,-5);
\draw (1,6) -- (1,-5);

\draw (-5,2.5) -- (3,2.5);
\draw (-5,-1) -- (3,-1);
\draw (-5,-4.5) -- (3,-4.5);

\begin{scope}[xshift=-0.6cm,yshift=3cm,scale=0.4]
\draw (0,0) grid[step=1] (3,6);
\draw (0,0) -- (3,3);
\draw (0,3) -- (3,6);
\fill (0.5,0.5) circle (0.2);
\fill (1.5,1.5) circle (0.2);
\fill (1.5,2.5) circle (0.2);
\fill (0.5,3.5) circle (0.2);
\fill (2.5,4.5) circle (0.2);
\fill (2.5,5.5) circle (0.2);
\end{scope}

\begin{scope}[xshift=1.4cm,yshift=3cm,scale=0.4]
\draw (0,0) grid[step=1] (3,6);
\draw (0,0) -- (3,3);
\draw (0,3) -- (3,6);
\fill (0.5,0.5) circle (0.2);
\fill (1.5,1.5) circle (0.2);
\fill (1.5,4.5) circle (0.2);
\fill (0.5,3.5) circle (0.2);
\fill (2.5,2.5) circle (0.2);
\fill (2.5,5.5) circle (0.2);
\end{scope}

\begin{scope}[xshift=-2.6cm,yshift=-0.5cm,scale=0.4]
\draw (0,0) grid[step=1] (3,6);
\draw (0,0) -- (3,3);
\draw (0,3) -- (3,6);
\fill (0.5,0.5) circle (0.2);
\fill (1.5,1.5) circle (0.2);
\fill (1.5,3.5) circle (0.2);
\fill (0.5,2.5) circle (0.2);
\fill (2.5,4.5) circle (0.2);
\fill (2.5,5.5) circle (0.2);
\end{scope}

\begin{scope}[xshift=-0.6cm,yshift=-0.5cm,scale=0.4]
\draw (0,0) grid[step=1] (3,6);
\draw (0,0) -- (3,3);
\draw (0,3) -- (3,6);
\fill (0.5,0.5) circle (0.2);
\fill (1.5,2.5) circle (0.2);
\fill (1.5,3.5) circle (0.2);
\fill (0.5,1.5) circle (0.2);
\fill (2.5,4.5) circle (0.2);
\fill (2.5,5.5) circle (0.2);
\end{scope}

\begin{scope}[xshift=1.4cm,yshift=-0.5cm,scale=0.4]
\draw (0,0) grid[step=1] (3,6);
\draw (0,0) -- (3,3);
\draw (0,3) -- (3,6);
\fill (0.5,0.5) circle (0.2);
\fill (1.5,3.5) circle (0.2);
\fill (1.5,4.5) circle (0.2);
\fill (0.5,1.5) circle (0.2);
\fill (2.5,2.5) circle (0.2);
\fill (2.5,5.5) circle (0.2);
\end{scope}

\begin{scope}[xshift=-2.6cm,yshift=-4cm,scale=0.4]
\draw (0,0) grid[step=1] (3,6);
\draw (0,0) -- (3,3);
\draw (0,3) -- (3,6);
\fill (0.5,0.5) circle (0.2);
\fill (1.5,1.5) circle (0.2);
\fill (1.5,4.5) circle (0.2);
\fill (0.5,2.5) circle (0.2);
\fill (2.5,3.5) circle (0.2);
\fill (2.5,5.5) circle (0.2);
\end{scope}

\begin{scope}[xshift=-0.6cm,yshift=-4cm,scale=0.4]
\draw (0,0) grid[step=1] (3,6);
\draw (0,0) -- (3,3);
\draw (0,3) -- (3,6);
\fill (0.5,0.5) circle (0.2);
\fill (1.5,2.5) circle (0.2);
\fill (1.5,4.5) circle (0.2);
\fill (0.5,1.5) circle (0.2);
\fill (2.5,3.5) circle (0.2);
\fill (2.5,5.5) circle (0.2);
\end{scope}

\draw (-4,4.25) node[scale=1.2] {$DC_{3,1}$};
\draw (-4,0.75) node[scale=1.2] {$DC_{3,2}$};
\draw (-4,-2.75) node[scale=1.2] {$DC_{3,3}$};

\draw (-2,-5) node[scale=1.2] {$DC'_{3,1}$};
\draw (0,-5) node[scale=1.2] {$DC'_{3,2}$};
\draw (2,-5) node[scale=1.2] {$DC'_{3,3}$};

\end{tikzpicture}
\end{center}
\caption{The partition of $DC_3$.}
\label{fig:repartitionDC3}
\end{figure}
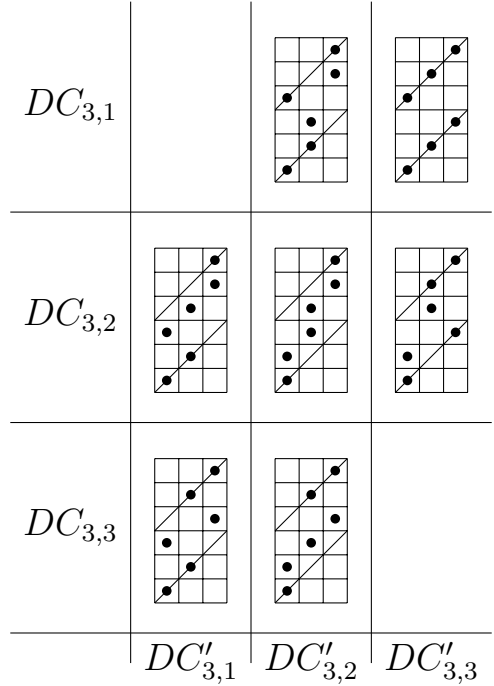

The combinatorial interpretations of Formulas (\ref{eq:hn1}) and (\ref{eq:symmetryofhnk}) in terms of Dellac configurations are simple. Every element of $DC_{n-1}$ can be obtained by deleting the $n$-th colomn (from left to right) and the $(n+1)$-th and $2n$-th rows (from bottom to top) of a unique element of $DC'_{n,n}$, which gives Formula (\ref{eq:hn1}). Afterwards, for all $D \in DC_{n,k} \cap DC'_{n,l}$,
\begin{itemize}
\item let $D^t \in DC_{n,l} \cap DC'_{n,k}$ be obtained by deleting the dots of the boxes $(k,n+1)$ and $(l,n)$ of $D$ and placing dots in the boxes $(l,n+1)$ and $(k,n)$,
\item let $D^r \in DC_{n,n+1-l} \cap DC'_{n,n-k+1}$ be obtained by rotating $D$ through $180^{\circ}$,
\end{itemize}
the maps $D \mapsto D^t$ and $D \mapsto D^r$ are involutions of $DC_n$ that induce bijections
\begin{align*}
DC_{n,k} \cap DC'_{n,l} &\longleftrightarrow DC_{n,l} \cap DC'_{n,k},\\
DC_{n,l} \cap DC'_{n,k} &\longleftrightarrow DC_{n,n-k+1} \cap DC'_{n,n+1-l},
\end{align*}
from which follows Formula (\ref{eq:symmetryofhnk}).

\subsection{Hetyei's model}
In his study of the alternation acyclic tournaments~\cite{Hetyei}, Hetyei proved that the median Genocchi number $H_{2n+1}$ is the number of pairs
$$((a_1,\hdots,a_n),(b_1,\hdots,b_n)) \in \mathbb{Z}^n \times \mathbb{Z}^n$$
such that $(a_i,b_i) \in [0,n] \times [n]$ for all $i \in [n]$, and the set $[n]$ is contained in the multiset $\{a_1,b_1,\hdots,a_n,b_n\}$. He then defined a free group action of $\left( \mathbb{Z} \backslash 2 \mathbb{Z} \right)^n$ on the set of these pairs, whose orbits are indexed by the $n$-tuples $(\{u_l,v_l\})_{l \in [n]}$ such that $(u_l,v_l) \in [l]^2$ for all $l \in [n]$ and the multiset $\{u_1,v_1,\hdots,u_n,v_n\}$ contains $[n]$, which raises a new proof of $H_{2n+1}$ being a multiple of $2^n$, and a new combinatorial model of $h_n$ through the set $\mathcal{M}_n$ of these tuples $(\{u_l,v_l\})_{l \in [n]}$. For example, the $h_3 = 7$ elements of $\mathcal{M}_3$ are 
\begin{align*}
&\{1,1\},\{2,2\},\{3,3\}\nonumber\\
&\{1,1\},\{1,2\},\{3,3\}\nonumber\\
&\{1,1\},\{2,2\},\{2,3\}\nonumber\\
&\{1,1\},\{1,2\},\{2,3\}\nonumber \\
&\{1,1\},\{1,1\},\{2,3\}\nonumber\\
&\{1,1\},\{2,2\},\{1,3\}\nonumber\\
&\{1,1\},\{1,2\},\{1,3\}\nonumber.
\end{align*}

It remains to connect $\mathcal{M}_n$ bijectively to the previous models of $h_n$. In Section \ref{sec:feigin}, we describe a model introduced by Feigin is his study of the degenerate flag varieties~\cite{Feigin}, and whose construction fits $\mathcal{M}_n$ in the best way. Incidentally, we define a slight adjustment of this model in a way that describes its inner construction. In Section \ref{sec:bijection}, we construct a bijection between Feigin's and Hetyei's model, which provides a combinatorial interpretation of the Kreweras triangle in terms of $\mathcal{M}_n$.

\section{Feigin's model}
\label{sec:feigin}

In order to label the torus fixed points of the degenerate flag variety $\mathcal{F}_n^a$, Feigin~\cite{Feigin} introduced the set $\mathcal{I}_n$ of the tuples $(I_0,\hdots,I_{n})$ where $I_i \subset~[n]$ has the conditions
\begin{align}
\#I_i &=i, \label{conditionI1}\\
I_{i-1} \backslash \{i\} &\subset I_{i}, \label{conditionI2}
\end{align}
In \cite[Proposition 3.1]{Feigin}, Feigin constructs a bijection $f_2 : \mathcal{I}_n \rightarrow DC_n$, thus $\# \mathcal{I}_n = h_n$. The set $\mathcal{I}_n$ can be partitionned into $\{\mathcal{I}_{n,k} : k \in [n]\}$ or $\{\mathcal{I}'_{n,l} : l \in [n]\}$ where $\mathcal{I}_{n,k}$ (respectively $\mathcal{I}'_{n,l}$) is the subset of the \linebreak elements $(I_0,\hdots,I_n) \in \mathcal{I}_n$ such that $k = \min \{i : 1 \in I_i\}$ (respectively $l = \min \{i : n \in I_i\}$). One can check that $f_2(\mathcal{I}_{n,k}) = DC_{n,k}$ and $f_2(\mathcal{I}'_{n,l}) = DC'_{n,l}$, so $\#\mathcal{I}_{n,k} = \# \mathcal{I}'_{n,k} = h_{n,k}$.  For example, the \linebreak $h_3 = 2+3+2$ elements of $\mathcal{I}_3$ are partitionned as depicted in Figure \ref{fig:repartitionI3}.

\begin{figure}[h]
\begin{center}

\begin{tikzpicture}[scale=1]

\draw (-4.6,3.5) -- (-4.6,-0);
\draw (-0.9,3.5) -- (-0.9,-0);
\draw (2.85,3.5) -- (2.85,-0);

\draw (-5.5,2.5) -- (6.5,2.5);
\draw (-5.5,1.5) -- (6.5,1.5);
\draw (-5.5,0.5) -- (6.5,0.5);

\draw (1,3) node[scale=1] {$\emptyset,\{1\},\{1,3\},[3]$};
\draw (4.75,3) node[scale=1] {$\emptyset,\{1\},\{1,2\},[3]$};
\draw (-2.75,2) node[scale=1] {$\emptyset,\{3\},\{1,3\},[3]$};
\draw (1,2) node[scale=1] {$\emptyset,\{2\},\{1,3\},[3]$};
\draw (4.75,2) node[scale=1] {$\emptyset,\{2\},\{1,2\},[3]$};
\draw (-2.75,1) node[scale=1] {$\emptyset,\{3\},\{2,3\},[3]$};
\draw (1,1) node[scale=1] {$\emptyset,\{2\},\{2,3\},[3]$};

\draw (-5.2,3) node[scale=1] {$\mathcal{I}_{3,1}$};
\draw (-5.2,2) node[scale=1] {$\mathcal{I}_{3,2}$};
\draw (-5.2,1) node[scale=1] {$\mathcal{I}_{3,3}$};

\draw (-2.75,0) node[scale=1] {$\mathcal{I}'_{3,1}$};
\draw (1,0) node[scale=1] {$\mathcal{I}'_{3,2}$};
\draw (4.75,0) node[scale=1] {$\mathcal{I}'_{3,3}$};

\end{tikzpicture}
\end{center}
\caption{The partition of $\mathcal{I}_3$.}
\label{fig:repartitionI3}
\end{figure}

In the following, we define a tweaking of this model.

\begin{nota}
For all $n$-tuple $(S_1,\hdots,S_n)$ of subsets of $[n]$ and for all $i \in [n]$, the set $\{j \in [n] : i \in S_j\}$ is denoted by $S_i^{-1}$.
\end{nota}

\begin{defi}
For all $n \geq 1$, let $\mathcal{S}_n$ be the set of the tuples $(S_1,\hdots,S_n)$ of subsets of $[n]$ with the conditions
\begin{itemize}
\item $\# S_i = \# S_i^{-1} = 1$ or $2$,
\item if $\# S_i = 2$, then $S_i^{-1} = \{i_1,i_2\}$ for some $i_1 < i < i_2$.
\end{itemize}
\end{defi}

\begin{rem}
We can partition $\mathcal{S}_n$ into $\{S_{n,k} : k \in [n]\}$ and $\{S'_{n,l} : l \in [n]\}$ where $S_{n,k}$ (respectively $S'_{n,k}$) is the set of the $(S_1,\hdots,S_n)$ such that $S_1^{-1} = \{k\}$ (respectively $S_n^{-1} = \{l\}$).
\end{rem}

\begin{prop}
The map $(I_i)_{i \in [0,n]} \mapsto (I_i \backslash I_{i-1})_{i \in [n]}$ is a bijection between $\mathcal{I}_n$ and $\mathcal{S}_n$, which sends $\mathcal{I}_{n,k}$ and $\mathcal{I}'_{n,l}$ to $\mathcal{S}_{n,k}$ and $\mathcal{S}'_{n,l}$ respectively. In particular $h_{n,k} = \# \mathcal{S}_{n,k} = \# \mathcal{S}'_{n,k}$.
\end{prop}

\begin{proof}
For all $i \in [n]$, let $S_i = I_i \backslash I_{i-1}$. There are two situations.
\begin{enumerate}
\item If $i \in I_{i-1} \cap I_i$ or $i \not\in I_{i-1}$, then $I_i = I_j \sqcup \{j\}$ for some $j \not\in [n]$, and $\# S_i = 1$.
\item Else $i \in I_{i-1}$ and $i \not\in I_i$, in which case $I_i = (I_{i-1} \backslash \{i\}) \sqcup \{j_1,j_2\}$ for some $(j_1,j_2) \in [n]^2$, and $\# S_i = 2$. Also, let 
\begin{align*}
i_1 &= \min \{j \in [n] : i \in I_j\} < i,\\
i_2 &= \min \{j \in [i,n] : i \in I_j\} > i,
\end{align*}
then $S_i^{-1} = \{i_1,i_2\}$. 
\end{enumerate}
So $(S_i)_{i \in [n]} \in \mathcal{S}_n$. The inverse map is obtained as follows. Let \linebreak $(S_i)_{i \in [n]}~\in~\mathcal{S}_n$ and $I_0 = \emptyset$. For all $i \in [n]$, suppose that we have defined $I_0,\hdots,I_{i-1}$ with the conditions (\ref{conditionI1}) and (\ref{conditionI2}), and the additional condition for all $j \in [n]$ :
\begin{equation}
\label{conditionI3}
\min \{ k \in [i-1] : j \in I_k\} = \min S_j^{-1}.
\end{equation}
If $\# S_i = 1$, then $I_i$ is defined as $I_{i-1} \sqcup S_i$. Otherwise $S_i^{-1} = \{i_1,i_2\}$ with $i_1 < i < i_2$, and $i \in I_{i_1}$ in view of condition (\ref{conditionI3}), hence $i \in I_{i-1}$, and $I_i$ is defined as $(I_{i-1} \backslash \{i\}) \sqcup S_i$. In both cases $I_0,\hdots,I_i$ have the conditions (\ref{conditionI1}),(\ref{conditionI2}) and (\ref{conditionI3}), and $(I_i)_{i \in [0,n]} \in \mathcal{I}_n$. The rest of the lemma is straightforward.
\end{proof}

\begin{rem}
For all $(S_i)_{i \in [n]} \in \mathcal{S}_n$, the inverse image $(I_i)_{i \in [0,n]}$ is also given by $I_i = \left( \bigcup_{j = 1}^i S_j \right) \backslash \{j \in [i] : \min S_j^{-1} < i < \max S_j^{-1}\}$.
\end{rem}

For example, the $h_3 = 2+3+2$ elements of $\mathcal{S}_3$ are partitionned as depicted in Figure \ref{fig:repartitionS3}.

\begin{figure}[h]
\begin{center}

\begin{tikzpicture}[scale=1]

\draw (-4.6,3.5) -- (-4.6,-0);
\draw (-0.9,3.5) -- (-0.9,-0);
\draw (2.85,3.5) -- (2.85,-0);

\draw (-5.5,2.5) -- (6.5,2.5);
\draw (-5.5,1.5) -- (6.5,1.5);
\draw (-5.5,0.5) -- (6.5,0.5);

\draw (1,3) node[scale=1] {$\{1\},\{3\},\{2\}$};
\draw (4.75,3) node[scale=1] {$\{1\},\{2\},\{3\}$};
\draw (-2.75,2) node[scale=1] {$\{3\},\{1\},\{2\}$};
\draw (1,2) node[scale=1] {$\{2\},\{1,3\},\{2\}$};
\draw (4.75,2) node[scale=1] {$\{2\},\{1\},\{3\}$};
\draw (-2.75,1) node[scale=1] {$\{3\},\{2\},\{1\}$};
\draw (1,1) node[scale=1] {$\{2\},\{3\},\{1\}$};

\draw (-5.2,3) node[scale=1] {$\mathcal{S}_{3,1}$};
\draw (-5.2,2) node[scale=1] {$\mathcal{S}_{3,2}$};
\draw (-5.2,1) node[scale=1] {$\mathcal{S}_{3,3}$};

\draw (-2.75,0) node[scale=1] {$\mathcal{S}'_{3,1}$};
\draw (1,0) node[scale=1] {$\mathcal{S}'_{3,2}$};
\draw (4.75,0) node[scale=1] {$\mathcal{S}'_{3,3}$};

\end{tikzpicture}
\end{center}
\caption{The partition of $\mathcal{S}_3$.}
\label{fig:repartitionS3}
\end{figure}

\begin{rem}
There is a natural injection $\mathfrak{S}_n \hookrightarrow \mathcal{S}_n : \sigma \mapsto (\{\sigma(i)\})_{i \in [n]}$, which is the analogous of the elements $(I_i)_{i \in [0,n]}$ with the conditions
\begin{align*}
\#I_i &= i,\\
I_{i-1} &\subset I_i
\end{align*}
forming a subset of $\mathcal{I}_n$ and labelling the torus fixed points of the flag variety $\mathcal{F}_n$~\cite{Feigin}.
\end{rem}

The bijection $\mathcal{S}'_{n,n} \rightarrow \mathcal{S}_{n-1}$, from which arises Formula (\ref{eq:hn1}), is the plain map $(S_1,\hdots,S_n) \mapsto (S_1,\hdots,S_{n-1})$. The involution $(S_1,\hdots,S_n) \in \mathcal{S}_n \mapsto (S_1^t,\hdots,S_n^t)$, defined by replacing every occurrence of $1$ (respectively $n$) by $n$ (respectively $1$) in all $S_i^t$, induces the bijection \linebreak$\mathcal{S}_{n,k} \cap \mathcal{S}'_{n,l} \rightarrow \mathcal{S}_{n,l} \cap \mathcal{S}'_{n,k}$. The involution $(S_1,\hdots,S_n)_n \rightarrow (S_1^r,\hdots,S_n^r)$, defined by $S_i^r = \{n+1-j : j \in S_{n+1-i}\}$, induces the bijection $\mathcal{S}_{n,k} \cap \mathcal{S}'_{n,l} \rightarrow \mathcal{S}_{n,n+1-l} \cap \mathcal{S}'_{n,n-k+1}$, from which follows Formula (\ref{eq:symmetryofhnk}).

\section{Bijective equivalence with Hetyei's model}
\label{sec:bijection}

\begin{defi}[map $\varphi : \mathcal{I}_n \rightarrow \mathcal{M}_n$]
Let $I = (I_0,\hdots,I_n) \in \mathcal{I}_n$ and $L_0 = (n,\hdots,1)$. Consider $k \in [n]$ and suppose that we have defined~:
\begin{itemize}
\item a multiset $\{u_{n-k+2},v_{n-k+2},\hdots,u_n,v_n\}$, such that $(u_l,v_l) \in [l]^2$ for all $l \in [n-k+2,n]$, which contains the set $[n-k+2,n]$;
\item a tuple $L_{k-1} = (j_1^{k-1},j_2^{k-1},\hdots,j_{n-k+1}^{k-1})$ such that $$\{j_1^{k-1},\hdots,j_{n-k+1}^{k-1}\} = [n] \backslash I_{k-1}.$$
\end{itemize}
We now define $(u_{n-k+1},v_{n-k+1}) \in [n-k+1]^2$ and $L_{k}$ as follows.

\begin{enumerate}[label=\arabic*.]

\item If $I_{k-1} \subset I_k$, let $p \in [n-k+1]$ such that $I_k = I_{k-1} \sqcup \{j_p^{k-1}\}$.

\begin{enumerate}[label=\alph*)]

\item If $k \in I_{k-1}$, we define $\{u_{n-k+1},v_{n-k+1}\}$ as $\{p,p\}$.

\item Otherwise, we define $\{u_{n-k+1},v_{n-k+1}\}$ as $\{p,n-k+1\}$. 
\end{enumerate}

In either case, let $$L_k = (j_1^{k-1},\hdots,j_{p-1}^{k-1},j_{n-k+1}^{k-1},j_{p+1}^{k-1},\hdots,j_{n-k}^{k-1}).$$

\item Otherwise $k \in I_{k-1}$ and $k \not\in I_k$, hence $I_k = (I_{k-1} \backslash \{k\}) \sqcup \{j_p^{k-1},j_q^{k-1}\}$ for some $1 \leq p < q \leq n-k+1$. We define $\{u_{n-k+1},v_{n-k+1}\}$ as $\{p,q\}$, and $$L_{k} = (j_1^{k-1},\hdots,j_{p-1}^{k-1},j_{n-k+1}^{k-1},j_{p+1}^{k-1},\hdots,j_{q-1}^{k-1},k,j_{q+1}^{k-1},\hdots,j_{n-k}^{k-1}).$$
\end{enumerate}
For the algorithm to move to $k+1$, we just need to show that \linebreak$n-k+1 \in \{u_{n-k+1},v_{n-k+1},\hdots,u_n,v_n\}$.
It is obvious if $\{u_{n-k+1},v_{n-k+1}\}$ is defined by Rule~1.b). Otherwise, by hypothesis, we have $k \in I_{k-1}$. Let $i_0 = \min \{i \in [n] : k \in I_i\} \in [k-1]$. By construction of $L_1,\hdots,L_{k-1}$, it is easy to see that $j_{n-k+1}^{i_0 - 1} = k$, hence $n-k+1 \in \{u_{n+1-i_0},v_{n+1-i_0}\}$ by either Rule 1.a) or Rule 2.

This algorithm provides a tuple $(\{u_l,v_l\})_{l \in [n]} \in \mathcal{M}_n$, that we denote by $\varphi(I)$.
\end{defi}

For example, let $I = (\emptyset,\{3\},\{1,3\},\{1,3,4\},\{1,2,3,5\},[5]) \in \mathcal{I}_5$ and $L_0 = 54321$. We obtain $\varphi(I) = (\{u_l,v_l\})_{l \in [5]}$ where
\begin{align*}
\{u_5,v_5\} &= \{3,5\}, L_1 = 5412 \text{ (rule 1.b))},\\
\{u_4,v_4\} &= \{3,4\}, L_2 = 542 \text{ (rule 1.b))},\\
\{u_3,v_3\} &= \{2,2\}, L_3 = 52 \text{ (rule 1.a))},\\
\{u_2,v_2\} &= \{1,2\}, L_4 = 4 \text{ (rule 2.)},\\
\{u_1,v_1\} &= \{1,1\}, L_5 = \varnothing  \text{ (rule 1.a))}.\\
\end{align*}

\begin{prop} The map $\varphi : \mathcal{I}_n \rightarrow \mathcal{M}_n$ is bijective. \end{prop}

\begin{proof}
We construct the inverse map of $\varphi$. Let $M = (\{u_l,v_l\})_{l \in [n]} \in \mathcal{M}_n$, $L_0 = (n,\hdots,1)$ and $I_0 = \emptyset$. Suppose that, for some $k \in [n]$, we defined subsets $I_0,\hdots,I_{k-1}$ of $[n]$ with conditions (\ref{conditionI1}) and (\ref{conditionI2}), and a tuple $L_{k-1} = (j_1^{k-1},\hdots,j_{n-k+1}^{k-1})$ with $\{j_1^k,\hdots,j_{n-k+1}^{k-1}\} = [n] \backslash I_{k-1}$. We define $I_k$ and $L_k$ as follows.

\begin{enumerate}[label=\Roman*.]

\item If $u_{n-k+1} = v_{n-k+1}$ or $n-k+1 \not\in \{u_{n-k+2},v_{n-k+2},\hdots,u_n,v_n\}$, there exists $p \in [n-k+1]$ such that $\{u_{n-k+1},v_{n-k+1}\} = \{p,p\}$ or $\{p,n-k+1\}$. We define $I_{k}$ as $I_{k-1} \sqcup \{j_p^{k-1}\}$, and $L_{k}$ as in Rule 1.

\item Otherwise $\{u_{n-k},v_{n-k}\} = \{p,q\}$ for some $1 \leq p < q \leq n-k+1$. We define $I_{k}$ as $( I_{k-1} \backslash \{k\} ) \sqcup \{j_p^{k-1},j_q^{k-1}\}$, and $L_{k}$ as in Rule~2.
\end{enumerate}

For the algorithm to iterate, we only need to prove that $\# I_k = k$ if it is defined by Rule II. In this context, let $n-i_0+1 = \max \{l \in [n] : n-k+1 \in \{u_l,v_l\}\}$, by hypothesis $i_0 \in [k-1]$. By construction of $L_1,\hdots,L_{k-1}$, we have $j_{n-1+k}^{i_0-1} = k$, hence $k \in I_{i_0}$, which implies that $k \in I_{k-1}$ in view of condition (\ref{conditionI2}).

So this algorithm provides an element $(I_0,\hdots,I_n) \in \mathcal{I}_n$ that we denote by $\phi(M)$, and it is straightforward that $\varphi$ and $\phi$ are inverse maps.
\end{proof}

\begin{defi}
Let $M = (\{u_l,v_l\})_{l \in [n]} \in \mathcal{M}_n$, we define a tuple $n = l_{1}>l_{2}>\hdots>l_{m} \geq 1$ as follows : if $u_{l_{i}} = v_{l_{i}} = l_i$, then $m$ is defined as $i$, otherwise we define $l_{i+1}$ as $\min \{u_{l_{i}},v_{l_{i}}\} < l_i$. This tuple is well-defined because $u_1 = v_1 = 1$ in general.

Afterwards, for all integer $l \in [l_{m},n]$, let $i \in [m]$ such that $l \in [l_{i},l_{i-1}-~1]$ (where $l_{0}$ is defined as $n+1$), we say that $l$ is \textit{$M$-redundant} if $l_{i} \in~\{u_l,v_l\}$. Note that the set of such integers is not empty because it contains $l_m$.
\end{defi}

We now define two partitions of $\mathcal{M}_n$, namely $\{\mathcal{M}_{n,k} : k \in [n]\}$ and $\{\mathcal{M}'_{n,l} : l \in [n]\}$, as follows.

\begin{defi}
For all $n \geq 1$ and $k \in [n]$, we define $\mathcal{M}_{n,k}$ (respectively $\mathcal{M}'_{n,l}$) as the set of the tuples $M \in \mathcal{M}_n$ such that
$$\max \{ i \in [n] : \text{ $i$ is $M$-redundant}\} = n-k+1$$
(respectively
$$\max \{ i \in [n] : 1 \in \{u_i,v_i\}\} = n-l+1).$$
\end{defi}

One can check that $\varphi(\mathcal{I}_{n,k}) = \mathcal{M}_{n,k}$ and $\varphi(\mathcal{I}'_{n,l}) = \mathcal{M}'_{n,l}$, hence $$\# \mathcal{M}_{n,k} = \# \mathcal{M}'_{n,k} = h_{n,k}.$$

For example, consider the tuple $$I = (\emptyset,\{3\},\{1,3\},\{1,3,4\},\{1,2,3,5\},[5]) \in \mathcal{I}_{5,2} \cap \mathcal{I}'_{5,4}$$ studied earlier, and its image $ M = \varphi(I) = (\{u_l,v_l\})_{l \in [5]}$. We can see in Picture \ref{fig:M} that $M \in \mathcal{M}_{5,2} \cap \mathcal{M}'_{5,4}$.

\begin{figure}[h]
\begin{center}

\begin{tikzpicture}[scale=1]

\draw (0,0) node[scale=1] {$\{\underline{1,1}\},$};
\draw (1.2,0) node[scale=1] {$\{\underline{1,2}\},$};
\draw (2.4,0) node[scale=1] {$\{2,2\},$};
\draw (3.6,0) node[scale=1] {$\{\underline{3,4}\},$};
\draw (4.8,0) node[scale=1] {$\{3,5\}$};

\draw (4.8,0) circle (0.5);
\draw (2.35,0) circle (0.5);
\draw (1.15,0) circle (0.5);
\draw (-0.05,0) circle (0.5);

\draw (4.8,0.5) to[bend right] (2.35,0.5);
\draw (2.35,0.5) to[bend right] (1.15,0.5);
\draw (1.15,0.5) to[bend right] (-0.05,0.5);

\end{tikzpicture}
\end{center}
\caption{The tuple $\varphi(M) \in \mathcal{M}_5$.}
\label{fig:M}
\end{figure}
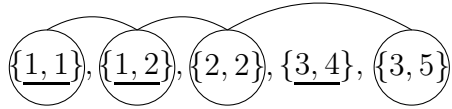

The $h_3=2+3+2$ elements of $\mathcal{M}_3$ are partitionned as depicted in Figure \ref{fig:repartitionM3}, which is the image of the partition of Figure \ref{fig:repartitionI3} by~$\varphi$.

\begin{figure}[h]
\begin{center}

\begin{tikzpicture}[scale=1]

\draw (-4.6,3.5) -- (-4.6,-0);
\draw (-0.9,3.5) -- (-0.9,-0);
\draw (2.85,3.5) -- (2.85,-0);

\draw (-5.5,2.5) -- (6.5,2.5);
\draw (-5.5,1.5) -- (6.5,1.5);
\draw (-5.5,0.5) -- (6.5,0.5);

\draw (1,3) node[scale=1] {$\{1,1\},\{1,2\},\{3,3\}$};
\draw (4.75,3) node[scale=1] {$\{1,1\},\{2,2\},\{3,3\}$};
\draw (-2.75,2) node[scale=1] {$\{1,1\},\{1,2\},\{1,3\}$};
\draw (1,2) node[scale=1] {$\{1,1\},\{1,2\},\{2,3\}$};
\draw (4.75,2) node[scale=1] {$\{1,1\},\{2,2\},\{2,3\}$};
\draw (-2.75,1) node[scale=1] {$\{1,1\},\{2,2\},\{1,3\}$};
\draw (1,1) node[scale=1] {$\{1,1\},\{1,1\},\{2,3\}$};

\draw (-5.2,3) node[scale=1] {$\mathcal{M}_{3,1}$};
\draw (-5.2,2) node[scale=1] {$\mathcal{M}_{3,2}$};
\draw (-5.2,1) node[scale=1] {$\mathcal{M}_{3,3}$};

\draw (-2.75,0) node[scale=1] {$\mathcal{M}'_{3,1}$};
\draw (1,0) node[scale=1] {$\mathcal{M}'_{3,2}$};
\draw (4.75,0) node[scale=1] {$\mathcal{M}'_{3,3}$};

\end{tikzpicture}
\end{center}
\caption{The partition of $\mathcal{M}_3$.}
\label{fig:repartitionM3}
\end{figure}


\begin{thebibliography}{9}

\bibitem{Barsky}
D. Barsky and D. Dumont. “Congruences pour les nombres de Genocchi de 2e espèce (French)”. In: \textit{Study Group on Ultrametric
Analysis} 34 (7th–8th years: 1979–1981), pp. 112–129.

\bibitem{Bigeni} 
A. Bigeni. “Combinatorial Study of Dellac Configurations and q-extended Normalized Median Genocchi Numbers”. In: \textit{Electronic
Journal of Combinatorics} 21.2 (2014), P2.32.

\bibitem{Dellac}
H. Dellac. “Problem 1735”. In: \textit{L’Intermédiaire des Mathématiciens} 7 (1900), pp. 9–10.

\bibitem{Dumont}
D. Dumont. “Interprétations combinatoires des nombres de Genocchi”. In: \textit{Duke Math}. J. 41 (1974), pp. 305–318.

\bibitem{DV}
D. Dumont and G. Viennot. “A combinatorial interpretation of
the Seidel generation of Genocchi numbers”. In: \textit{Ann. Discrete
Math}. 6 (1980), pp. 77–87.

\bibitem{Feigin}
E. Feigin. “Degenerate flag varieties and the median Genocchi
numbers”. In: \textit{Math. Res. Lett.} 18.6 (2011), pp. 1163–1178.

\bibitem{HZ}
G.-N. Han and J. Zeng. “On a q-sequence that generalizes the median Genocchi numbers”. In: \textit{Ann. Sci. Math. Québec} 23.1 (1999),
pp. 63–72.

\bibitem{Hetyei}
Gábor Hetyei. \textit{Alternation acyclic tournaments}. 2017. eprint: arXiv:
1704.07245.

\bibitem{G}
OEIS Foundation Inc., The On-Line Encyclopedia of Integer Sequences, \url{http://oeis.org/A110501} (2011).

\bibitem{H}
OEIS Foundation Inc., The On-Line Encyclopedia of Integer Sequences, \url{http://oeis.org/A005439} (2011).

\bibitem{h}
OEIS Foundation Inc., The On-Line Encyclopedia of Integer Sequences, \url{http://oeis.org/A000366} (2011).

\bibitem{Kreweras}
G. Kreweras. “Sur les permutations comptées par les nombres de
Genocchi de première et deuxième espèce (French)”. In: \textit{European
J. Combin.} 18.1 (1997), pp. 49–58.

\bibitem{KB}
G. Kreweras and J. Barraud. “Anagrammes alternés”. In: \textit{European J. Combin.} 18.8 (1997), pp. 887–891.



\end{thebibliography}
\end{document}